\documentclass{amsart}
\usepackage{bbm}
\usepackage{mathabx}
\usepackage{tensor}
\usepackage{amssymb}
\usepackage{amsmath}
\usepackage{amsfonts}
\usepackage{amsthm}
\usepackage{stmaryrd}
\usepackage[all]{xy}
\usepackage{mathrsfs}
\usepackage{graphicx}
\usepackage{hyperref}
\usepackage{color}
\usepackage{tikz-cd}
\usepackage{tkz-euclide}
\usepackage{tikz}
\usetikzlibrary{matrix}
\usepackage{enumitem}
\usepackage{amscd}
\usepackage{pifont}
\usetikzlibrary{arrows}
\usepackage{placeins}
\usepackage{xspace}
\usetikzlibrary{calc,intersections,through,backgrounds}

\setlist[enumerate]{label=(\arabic*),nosep}

\newtheorem{thm}{Theorem}[section]
\newtheorem{theorem}[thm]{Theorem}

\newtheorem{corollary}[thm]{Corollary}
\newtheorem{lemma}[thm]{Lemma}
\newtheorem{proposition}[thm]{Proposition}
\theoremstyle{definition}

\newtheorem{definition}[thm]{Definition}
\newtheorem{example}[thm]{Example}

\newtheorem{notation}[thm]{Notation}

\newtheorem{remark}[thm]{Remark}

\numberwithin{equation}{section}

\def\FF{\mathbb{F}}
\def\QQ{\mathbb{Q}}
\def\ZZ{\mathbb{Z}}
\def\RR{\mathbb{R}}
\def\NP{\mathrm{NP}}
\def\GNP{\mathrm{GNP}}
\def\HP{\mathrm{HP}}

\DeclareMathOperator{\Cone}{Cone}
\DeclareMathOperator{\conv}{conv}
\DeclareMathOperator{\Vol}{Vol}
\newcommand{\Tr}{\mathrm{Tr}}
\DeclareMathOperator{\Vertices}{Vert}
\def\sgn{\mathrm{sgn}}
\newcommand{\xx}{\underline x}
\def\aa{\underline a}

\def\MM{\mathbb{M}}
\def\TT{\mathbb{T}}
\def\SS{\mathbb S}
\begin{document}\large
	\title{Asymptotic low slopes of toric exponential sums}
	\author{Rufei Ren}
	\address{School of Mathematical Sciences, Fudan University, China}
	\date{\today}

\begin{abstract}
Let $\Delta$ be an integral Newton polytope containing the origin.
We prove that the generic Newton slopes corresponding to Hodge slopes
strictly smaller than one differ from their Hodge values by
$O_\Delta(p^{-1})$. No hypothesis is imposed on the semigroup generated
by the lattice points of $\Delta$. The proof combines an auxiliary
prime and the Cartier operator to show that a coefficient of each
required Dwork principal minor is nonzero. When the origin is interior
to $\Delta$, duality gives the same estimate for the corresponding
high slopes and, in dimension two, for all slopes.
\end{abstract}
\maketitle
\setcounter{tocdepth}{1}
\tableofcontents

\section{Introduction}

Let $\Delta\subset\RR^n$ be a full-dimensional integral convex polytope
containing the origin. For a Laurent polynomial
$f=\sum_P a_P\xx^P$, we use the Newton polytope at infinity
\[
\Delta(f):=\conv\bigl(\{0\}\cup\{P:a_P\ne0\}\bigr).
\]
For $f$ over $\FF_q$, where $q:=p^a$, fix a nontrivial additive character
$\chi\colon\FF_p\longrightarrow\overline{\QQ}_p^{\times}$ and put
\begin{gather*}
 S_k^*(f,\chi):=\sum_{\xx\in(\FF_{q^k}^{\times})^n}
 \chi\bigl(\Tr_{\FF_{q^k}/\FF_p}(f(\xx))\bigr)\\
\text{and}\quad
L_f^*(\chi,s):=\exp\left(\sum_{k\ge1}S_k^*(f,\chi)\frac{s^k}{k}\right).
\end{gather*}
For $f$ nondegenerate with Newton polytope $\Delta$, the function
\[
 L_f^*(\chi,s)^{(-1)^{n-1}}
\]
is a polynomial of degree $V_\Delta:=n!\Vol(\Delta)$.
Since $\chi$ is fixed, we omit it from the notation for these functions
and for the characteristic series below. For a power series $G(s)$
with constant term $1$, we denote its Newton polygon by $\NP_{v_q}(G)$,
where the valuation $v_q$ is normalized by $v_q(q)=1$.

Put
\[
 \Cone(\Delta):=\RR_{\ge0}\Delta
\quad\text{and}\quad
\MM(\Delta):=\Cone(\Delta)\cap\ZZ^n,
\]
and, for $P\in\Cone(\Delta)$, define
\[
 w(P):=\inf\{t\ge0:P\in t\Delta\}.
\]
Let $D$ be the least positive integer such that
$w(\MM(\Delta))\subset D^{-1}\ZZ$. For $k\in\ZZ$, put
\[
 W_\Delta(k):=\#\{P\in\MM(\Delta):w(P)=k/D\},
\]
and define
\[
 h_\Delta(k):=\sum_{j=0}^n(-1)^j\binom njW_\Delta(k-jD).
\]
The numbers $h_\Delta(k)$ are nonnegative integers, vanish outside
$0\le k\le nD$ and have sum $V_\Delta$ \cite{AS89}.
The Hodge polygon $\HP_L(\Delta)$ is the polygon starting at $(0,0)$
with slope $k/D$ of multiplicity $h_\Delta(k)$, in increasing order
of $k$. For nondegenerate $f$ with Newton polytope $\Delta$, the
polygon $\NP_{v_q}\bigl(L_f^*(s)^{(-1)^{n-1}}\bigr)$ lies above $\HP_L(\Delta)$ and has the same
endpoints \cite{AS89,DL91}.

For every sufficiently large prime $p$, the space of nondegenerate
Laurent polynomials with Newton polytope $\Delta$ is nonempty
\cite{AS89}. Whenever this space is nonempty,
we denote by $\GNP_{L,p}(\Delta)$ the generic Newton polygon of $L_f^*(s)^{(-1)^{n-1}}$
as $f$ runs over all nondegenerate Laurent polynomials over
$\overline{\FF}_p$ with Newton polytope $\Delta$.
Let
\[
 \beta_1\le\cdots\le\beta_{V_\Delta}
\quad\text{and}\quad
\alpha_{p,1}\le\cdots\le\alpha_{p,V_\Delta}
\]
be the slopes of $\HP_L(\Delta)$ and $\GNP_{L,p}(\Delta)$, respectively.

Put
\[
m_\Delta:=\#\{P\in\MM(\Delta):w(P)<1\}.
\]
The first $m_\Delta$ slopes of $\HP_L(\Delta)$ are the lattice
weights strictly smaller than one, counted with multiplicity.
Our main result concerns these slopes. In all limits below,
$p$ runs through the prime numbers.

\begin{theorem}\label{thm:main-L}
There are constants $C_\Delta\ge0$ and $p_0(\Delta)$ such that, for
every prime $p>p_0(\Delta)$, we have
\begin{equation}\label{eq:quantitative-L-height}
0\le\GNP_{L,p}(\Delta)(x)-\HP_L(\Delta)(x)
\le\frac{C_\Delta}{p-1}
\qquad(0\le x\le m_\Delta).
\end{equation}
In particular,
\begin{equation}\label{eq:quantitative-L-slope}
|\alpha_{p,i}-\beta_i|\le\frac{C_\Delta}{p-1}
\qquad(1\le i\le m_\Delta).
\end{equation}
For each such $p$, there is a nondegenerate Laurent polynomial $f_p$
over $\FF_{p^2}$ with Newton polytope $\Delta$ for which
\eqref{eq:quantitative-L-height} also holds with $\GNP_{L,p}(\Delta)$
replaced by $\NP_{v_{p^2}}\bigl(L_{f_p}^*(s)^{(-1)^{n-1}}\bigr)$.
\end{theorem}

Consequently,
\begin{equation}\label{eq:main-L-height}
\lim_{p\to\infty}\GNP_{L,p}(\Delta)=\HP_L(\Delta)
\end{equation}
uniformly on $[0,m_\Delta]$, and
$\lim_{p\to\infty}\alpha_{p,i}=\beta_i$ for
$1\le i\le m_\Delta$.

When the origin is interior to $\Delta$, duality gives the following
consequence for the high slopes.

\begin{corollary}\label{cor:interior}
Suppose that the origin is interior to $\Delta$.
With the constants in Theorem~\ref{thm:main-L}, for every prime
$p>p_0(\Delta)$, \eqref{eq:quantitative-L-slope} holds for every $i$
with $\beta_i>n-1$.
If $n=2$, then \eqref{eq:quantitative-L-height} holds on
$[0,V_\Delta]$ and \eqref{eq:quantitative-L-slope} holds for every
$1\le i\le V_\Delta$. In particular, \eqref{eq:main-L-height} holds
uniformly on $[0,V_\Delta]$.
\end{corollary}

For a general integral polytope, the conclusion of
Theorem~\ref{thm:main-L} does not extend to all slopes.
For example, the three-dimensional simplex
\[
 \conv\{0,e_1,e_2,(1,1,3)\}
\]
has Hodge slopes $0$, $4/3$ and $5/3$, whereas its generic Newton slopes are
$0$, $3/2$ and $3/2$ whenever $p\equiv2\pmod3$.
We give the calculation and the analogous examples in higher dimensions
in Section~\ref{sec:obstruction}.

Adolphson and Sperber established the Hodge bound and the cohomological
description of nondegenerate exponential sums \cite{AS89}.
They asked whether the generic Newton polygon equals the Hodge polygon
when $p$ satisfies a suitable congruence condition. Wan developed
decomposition methods for this question and obtained generic
ordinarity in a suitable residue class depending on $\Delta$
\cite{Wan93,Wan04}.

Wan proposed the one-variable asymptotic conjecture in 2000.
Scholten and Zhu proved its first-slope case \cite{SZ02}.
Zhu subsequently proved full one-variable asymptotic
results for generic polynomials over $\QQ$ and over $\overline{\QQ}$
\cite{Zhu03,Zhu04}. Blache obtained multivariable asymptotic results for free sums of
intervals along a lattice basis \cite{Blache11}.

For the general asymptotic conjecture, we put
\[
 E_\Delta:=\left\{\sum_{P\in\Delta\cap\ZZ^n}m_PP:
                  m_P\in\ZZ_{\ge0}\right\}.
\]
Under the hypothesis that $\MM(\Delta)\setminus E_\Delta$ is finite,
Wan \cite[Conjecture~1.11]{Wan04} conjectured that
\eqref{eq:main-L-height} holds on the full interval $[0,V_\Delta]$.
Zhu \cite[Theorem~1.10]{Zhu12} obtained convergence on the full
interval under this hypothesis. This includes all lattice polygons:
every lattice polygon has a unimodular triangulation, and its lattice
points generate $\MM(\Delta)$.
Theorem~\ref{thm:main-L} gives a uniform error of order $p^{-1}$
for the slopes corresponding to Hodge slopes strictly smaller than
one, without a hypothesis on $E_\Delta$.
Example~\ref{ex:nontrivial-low-slopes} gives a three-dimensional simplex
for which $\MM(\Delta)\setminus E_\Delta$ is infinite and the theorem
applies to a nonzero Hodge slope $1/2$.

Related conjectures concern reductions of a fixed generic polynomial
over a number field \cite[Conjecture~1.12]{Wan04} and
$T$-adic characteristic series \cite[Conjecture~7.11]{LW09}.

For Artin--Schreier--Witt towers, Zhu obtained explicit generic slopes
for rectangular polytopes in two variables \cite{Zhu16}. In earlier
work, we studied parallelotope bases and the propagation of contact
points at integral weights for triangular bases \cite{Ren20,Ren23}.
Here we consider additive characters of order $p$ and study
the Newton polygons as $p$ tends to infinity.

Dwork's trace formula reduces the estimate to finitely many principal
minors, each indexed by all lattice points up to a fixed weight. For each residue class
of $p$ modulo the coordinate denominator, we choose a fixed auxiliary
prime $\ell$ in the same class. A calculation with the Cartier operator
at $\ell$ proves that a rational determinant coefficient is nonzero.
This yields a $p$-adic unit coefficient in the corresponding principal
minor for every sufficiently large $p$ in that class.

Below weight one, the chosen coefficient has vertex multiplicities
less than $p$ and other multiplicities at most $\ell-1$.
Hence only the coefficients $1/j!$ with $j<p$ of the Artin--Hasse
exponential occur. The resulting nonvanishing conditions have total degree
$O_\Delta(p)$ and can be satisfied simultaneously over $\FF_{p^2}$.
Convexity then gives the estimate throughout $[0,m_\Delta]$.

\section{Preliminaries and Dwork's trace formula}
\label{sec:dwork}

\subsection{Dwork's trace formula}
\label{sec:trace-formula}

\begin{definition}
For a face $\delta$ of $\Delta$, put
$f_\delta:=\sum_{P\in\delta\cap\ZZ^n}a_P\xx^P$.
A Laurent polynomial with Newton polytope $\Delta$ is nondegenerate
if, for every face $\delta$ not containing zero, the equations
\[
 x_1\frac{\partial f_\delta}{\partial x_1}
 =\cdots=
 x_n\frac{\partial f_\delta}{\partial x_n}=0
\]
have no common solution in $(\overline{\FF}_p^{\times})^n$.
\end{definition}

\begin{notation}\label{not:polygons}
For $t\ge0$, put
\begin{gather*}
\TT_{<t}:=\{P\in\MM(\Delta):w(P)<t\}\\
\text{and}\quad
\TT_{\le t}:=\{P\in\MM(\Delta):w(P)\le t\}.
\end{gather*}
We denote by $\HP_C(\Delta)$ the infinite convex polygon starting at
$(0,0)$ with slopes $w(P)$ for $P\in\MM(\Delta)$, counted with
multiplicity and arranged in nondecreasing order.
\end{notation}

For $k<D$, we have $h_\Delta(k)=W_\Delta(k)$. Hence
\begin{equation}\label{eq:low-Hodge-equal}
 \HP_L(\Delta)(x)=\HP_C(\Delta)(x)\qquad(0\le x\le m_\Delta).
\end{equation}
The characteristic series is
\[
 C_f^*(s):=\exp\left(-\sum_{k\ge1}
       \frac{S_k^*(f)}{(q^k-1)^n}\frac{s^k}{k}\right).
\]
We denote by $\GNP_{C,p}(\Delta)$ the generic Newton polygon of
$C_f^*$ with respect to $v_q$. For characters of conductor $p$,
the characteristic-series polygons in \cite{Ren20,Ren23} have heights
$p-1$ times those used here.
For nondegenerate $f$, the relation with the finite polynomial is
\begin{equation}\label{eq:C-L-product}
 C_f^*(s)=\prod_{j\ge0}L_f^*(q^js)^{(-1)^{n-1}\binom{n+j-1}{n-1}}.
\end{equation}
Hence the slopes of $C_f^*$ are obtained by adding $j$ to each slope
of $L_f^*(s)^{(-1)^{n-1}}$, with multiplicity $\binom{n+j-1}{n-1}$ for every $j\ge0$
\cite[Lemma~2.7]{LW09}. In particular, the slopes strictly below one
come from the factor $L_f^*(s)^{(-1)^{n-1}}$ alone.

We now recall the matrix used in Dwork's trace formula
\cite[Sections~4 and 5]{LW09}. Let
\[
 E_p(z):=\exp\left(\sum_{i\ge0}\frac{z^{p^i}}{p^i}\right)
       \in\ZZ_p[[z]]
\]
be the Artin--Hasse exponential. Its coefficient of $z^k$ is $1/k!$ for $0\le k<p$. Choose a Dwork parameter $\pi$
with $E_p(\pi)=\chi(1)$ and normalize $v_\pi(\pi)=1$.
For $q=p^a$, the valuations satisfy $v_\pi=a(p-1)v_q$.
For Teichm\"uller lifts $\widehat a_R$, put
\begin{gather*}
 E_f(\xx):=\prod_{R\in\Delta\cap\ZZ^n}
 E_p(\pi\widehat a_R\xx^R)
 =\sum_U e_U\xx^U\\
\text{and}\quad
N_{P,Q}:=e_{pP-Q}\quad(P,Q\in \MM(\Delta)).
\end{gather*}
The coefficient of each power of $\pi$ in this product is a Laurent
polynomial. Hence the product is well defined as a formal series in
$\pi$, even when the cone contains lines. The weight inequality gives
\begin{equation}\label{eq:entry-bound}
 v_\pi(N_{P,Q})\ge w(pP-Q)\ge p w(P)-w(Q).
\end{equation}
An entry with $pP-Q$ outside the cone is zero.

After adjoining a $D$th root of $\pi$, use the weighted matrix
\[
 \mathcal D_{P,Q}:=\pi^{w(Q)-w(P)}N_{P,Q}.
\]
Then
\begin{equation}\label{eq:nuclear-bound}
 v_\pi(\mathcal D_{P,Q})\ge(p-1)w(P)\longrightarrow+\infty
 \qquad(w(P)\longrightarrow\infty),
\end{equation}
uniformly in $Q$. Therefore $\mathcal D$ is completely continuous and has the same
principal minors as $N$. By \eqref{eq:nuclear-bound}, we may compute
the determinant expansions on finite truncations and pass to the limit.

For the rest of this subsection, put $q:=p^2$.
Let $\sigma_{\mathrm{Frob}}$ act as Frobenius on the unramified
coefficient field and fix $\pi$. Dwork's trace formula gives
\begin{equation}\label{eq:trace-q2}
 C_f^*(s)=\det(1-s\mathcal D^{\sigma_{\mathrm{Frob}}}\mathcal D).
\end{equation}
The ordinary $q$-adic height is $v_\pi/(2(p-1))$ in this formula.

\begin{lemma}\label{lem:minor-isolation}
Let $0\le t<1$ be a lattice weight and put
\[
 \TT:=\TT_{\le t},\quad
m:=\#\TT,\quad
H:=\sum_{P\in \TT}w(P)
\quad\text{and}\quad
\eta:=\min_{P\notin \TT}w(P)-t.
\]
Then $\eta>0$. Suppose that $f$ is defined over $\FF_q$ and that
\[
 v_\pi\det(N_{P,Q})_{P,Q\in \TT}\le(p-1)H+C.
\]
If $(p-1)\eta>2C$, then
\[
 H\le\NP_{v_q}(C_f^*)(m)\le H+\frac C{p-1}.
\]
\end{lemma}

\begin{proof}
For a finite index set $I$, write $W(I):=\sum_{P\in I}w(P)$.
The set $\TT$ is the unique set of cardinality $m$ of smallest weight.
Every other such set satisfies $W(I)\ge H+\eta$.
By Cauchy--Binet, the coefficient of $(-s)^m$ in
\eqref{eq:trace-q2} is
\[
 \sum_{\#I=\#J=m}
 \det(\mathcal D^{\sigma_{\mathrm{Frob}}}_{I,J})\det(\mathcal D_{J,I}).
\]
The summand $I=J=\TT$ has valuation
$2v_\pi\det N_{\TT,\TT}\le2(p-1)H+2C$.
Every other summand has valuation at least
\[
 (p-1)(W(I)+W(J))\ge2(p-1)H+(p-1)\eta.
\]
Since $(p-1)\eta>2C$, the summand with $I=J=\TT$ has strictly
smaller valuation than every other summand. Dividing by $2(p-1)$ gives
the upper bound. The lower bound follows from
\cite[Theorem~2.8]{LW09}.
\end{proof}

\begin{lemma}\label{lem:formal-specialization}
Let $G(\pi)=\sum_{j\ge j_0}g_j\pi^j$ have coefficients in an
unramified extension of $\ZZ_p$. Suppose that $g_{j'}$ is a unit and
$j'-j_0<p-1$. Then $v_\pi(G(\pi))\le j'$.
\end{lemma}
\begin{proof}
Let $e\le j'$ be the smallest index with $g_e$ a unit.
For $j<e$, the coefficient $g_j$ is divisible by $p$, and hence
$v_\pi(g_j\pi^j)\ge j_0+p-1>j'$.
Terms with $j>e$ have valuation greater than $e$.
Therefore $v_\pi(G(\pi))=e\le j'$.
\end{proof}

List the distinct lattice weights below one as
\[
 0=\rho_0<\rho_1<\cdots<\rho_K<1.
\]
For $0\le k\le K$, put
\[
 m_k:=\#\TT_{\le\rho_k}
 \quad\text{and}\quad
 H_k:=\sum_{P\in\TT_{\le\rho_k}}w(P).
\]
Then $m_0=1$, $H_0=0$ and $m_K=m_\Delta$.

\begin{proposition}\label{prop:dwork-reduction}
Let $C\ge0$ and suppose that $p-1>2DC$.
Let $f$ be a nondegenerate Laurent polynomial over $\FF_q$
with Newton polytope $\Delta$. Suppose that
\begin{equation}\label{eq:required-minors}
 v_\pi\det N_{\TT_{\le\rho_k},\TT_{\le\rho_k}}
 \le(p-1)H_k+C
 \qquad(1\le k\le K).
\end{equation}
Then, for $0\le x\le m_\Delta$, we have
\begin{equation}\label{eq:C-uniform}
 \begin{gathered}
 0\le\NP_{v_q}(C_f^*)(x)-\HP_C(\Delta)(x)
 \le\frac{C}{p-1}\\
 \text{and}\quad
 0\le\NP_{v_q}\bigl(L_f^*(s)^{(-1)^{n-1}}\bigr)(x)
       -\HP_L(\Delta)(x)
 \le\frac{C}{p-1}.
 \end{gathered}
\end{equation}
The same inequalities hold for the corresponding generic Newton
polygons. In particular,
\[
 |\alpha_{p,i}-\beta_i|\le\frac{C}{p-1}
 \qquad(1\le i\le m_\Delta).
\]
\end{proposition}

\begin{proof}
Put $\varepsilon:=C/(p-1)$.
Since all lattice weights belong to $D^{-1}\ZZ$, the gap after
each $\rho_k$ is at least $1/D$.
Lemma~\ref{lem:minor-isolation} gives
\begin{equation}\label{eq:cutoff-estimate}
 H_k\le\NP_{v_q}(C_f^*)(m_k)
 \le H_k+\varepsilon
 \qquad(1\le k\le K).
\end{equation}
The entry $\mathcal D_{0,0}$ is congruent to $1$ modulo the
maximal ideal. Every other entry has positive valuation.
For the nonzero rows this follows from \eqref{eq:nuclear-bound},
and for the origin row it follows from
$\mathcal D_{0,Q}=\pi^{w(Q)}e_{-Q}$ for $Q\ne0$.
Hence \eqref{eq:trace-q2} gives
$C_f^*(s)\equiv1-s$ modulo the maximal ideal.
Its Newton polygon therefore has height zero on $[0,1]$.

For $x=(1-\theta)m_{k-1}+\theta m_k$, where $0\le\theta\le1$,
convexity and \eqref{eq:cutoff-estimate} give
\begin{align*}
 \NP_{v_q}(C_f^*)(x)
 &\le(1-\theta)\NP_{v_q}(C_f^*)(m_{k-1})
       +\theta\NP_{v_q}(C_f^*)(m_k)\\
 &\le(1-\theta)H_{k-1}+\theta H_k+\varepsilon\\
 &=\HP_C(\Delta)(x)+\varepsilon.
\end{align*}
Together with the Hodge lower bound, this proves the first
inequality. It also covers $K=0$, when $m_\Delta=1$.

Since $\rho_K\le1-1/D$, the last of the first
$m_\Delta$ characteristic-series slopes is at most
\[
 \NP_{v_q}(C_f^*)(m_\Delta)
 -\NP_{v_q}(C_f^*)(m_\Delta-1)
 \le\rho_K+\varepsilon<1.
\]
Every factor indexed by $j\ge1$ in
\eqref{eq:C-L-product} has slopes at least one.
Thus the first $m_\Delta$ slopes of $C_f^*$ are exactly the
first $m_\Delta$ slopes of $L_f^*(s)^{(-1)^{n-1}}$.
The second inequality follows from \eqref{eq:low-Hodge-equal}.

The generic Newton polygons lie below those of $f$ and above
the corresponding Hodge polygons. Hence they satisfy the same
bounds. For $0\le i\le m_\Delta$, put
$e_i:=\GNP_{L,p}(\Delta)(i)-\HP_L(\Delta)(i)$.
Then
\[
 0\le e_i\le\varepsilon
 \quad\text{and}\quad
 \alpha_{p,i}-\beta_i=e_i-e_{i-1}\quad(i\ge1).
\]
This proves the slope estimate.
\end{proof}

\subsection{A simplicial subdivision}
\label{sec:preliminaries}

\begin{notation}\label{not:subdivision}
Put $V:=\Vertices(\Delta)\setminus\{0\}$.
Consider the fan consisting of the cones over the facets of $\Delta$
not containing $0$ and all their faces, as in the facial decomposition
of \cite[Section~3.1]{Wan04}.
Fix a simplicial refinement $\Sigma$ with the same rays such that
the morphism of toric varieties induced by the refinement is projective
\cite[Proposition~11.1.7]{CLS11}.
Thus every maximal cone of $\Sigma$ is generated by $n$ linearly
independent vertices in $V$, the fan contains all faces of its cones
and its union is $\Cone(\Delta)$.

For $P\in\Cone(\Delta)$, let $\sigma_P$ be the unique cone of
$\Sigma$ whose relative interior contains $P$. Define the coordinates
$\lambda_v(P)$ by
\[
P=\sum_{v\in V}\lambda_v(P)v,\quad
\lambda_v(P)\ge0
\quad\text{and}\quad
\lambda_v(P):=0\quad(v\notin\sigma_P).
\]
These coordinates are unique and agree on common faces. Since every
cone lies over a facet on which $w=1$, we have
\[
\sum_{v\in V}\lambda_v(P)=w(P).
\]
Let $D_\Sigma$ be the least positive integer such that
\[
D_\Sigma\lambda_v(P)\in\ZZ
\qquad(P\in\MM(\Delta),\ v\in V).
\]
Such an integer exists by Cramer's rule on each maximal cone, and
$D\mid D_\Sigma$.
\end{notation}

\begin{lemma}\label{lem:cone-weights}
There is a function $\psi\colon\Cone(\Delta)\longrightarrow\RR_{\ge0}$
with the following properties.
\begin{enumerate}[leftmargin=*,nosep]
\item On each cone of $\Sigma$, the function $\psi$ is the restriction
of a linear form with rational coefficients. In particular,
\begin{equation}\label{eq:psi-definition}
\psi(P)=\sum_{v\in V}\psi(v)\lambda_v(P),
\end{equation}
and $\psi(P)>0$ for $P\ne0$.
\item The function $\psi$ is convex and homogeneous, and
\begin{equation}\label{eq:psi-properties}
\psi(P+Q)\le\psi(P)+\psi(Q)
\qquad(P,Q\in\Cone(\Delta)).
\end{equation}
\item If $P=X_1+\cdots+X_a$ with $X_i\in\Cone(\Delta)$, then
\begin{equation}\label{eq:cone-equality}
\psi(P)=\sum_{i=1}^a\psi(X_i)
\quad\Longleftrightarrow\quad
X_i\in\sigma_P\quad(1\le i\le a).
\end{equation}
\end{enumerate}
\end{lemma}

\begin{proof}
(1) Taking the negative of the divisor support function supplied by
\cite[Theorem~7.2.12]{CLS11}, we obtain rational numbers $h_v$
and linear forms $h_\sigma$ with the following property.
If an $n$-dimensional cone $\sigma\in\Sigma$ lies in the cone over a
facet $F$ of $\Delta$ not containing $0$, then
\[
h_\sigma(v)=h_v\quad(v\in V\cap\sigma)
\quad\text{and}\quad
h_\sigma(v)<h_v\quad(v\in V\cap(F\setminus\sigma)).
\]
Let $w_F$ be the linear form defining $F$, normalized by $w_F|_F=1$.
For a positive rational number $\varepsilon$, put
\[
c_v:=1+\varepsilon h_v
\quad\text{and}\quad
L_\sigma:=w_F+\varepsilon h_\sigma.
\]
Then
\[
c_v-L_\sigma(v)
=1-w_F(v)+\varepsilon\bigl(h_v-h_\sigma(v)\bigr).
\]
For $v\in F\setminus\sigma$, the first difference is zero and the
second is positive. For $v\notin F$, we have $1-w_F(v)>0$.
Since there are finitely many pairs $(\sigma,v)$, a sufficiently
small $\varepsilon>0$ gives $c_v>0$ and
\begin{equation}\label{eq:strict-vertex-inequality}
L_\sigma(v)=c_v\quad(v\in V\cap\sigma)
\quad\text{and}\quad
L_\sigma(v)<c_v\quad(v\in V\setminus\sigma).
\end{equation}
Define
\[
\psi(P):=\sum_{v\in V}c_v\lambda_v(P).
\]
The coordinates agree on common faces, hence $\psi$ is linear with
rational coefficients on every cone of $\Sigma$. Also $\psi(v)=c_v$ and
\[
\psi(P)\ge\min_{v\in V}c_v\,w(P)>0\qquad(P\ne0).
\]

(2) For every $n$-dimensional cone $\sigma\in\Sigma$ and every
$P\in\Cone(\Delta)$,
\[
\psi(P)-L_\sigma(P)
=\sum_{v\in V}\lambda_v(P)\bigl(c_v-L_\sigma(v)\bigr)\ge0.
\]
Equality holds precisely when every vertex with $\lambda_v(P)>0$
belongs to $\sigma$, or equivalently when $P\in\sigma$.
Therefore
\begin{equation}\label{eq:psi-linear-forms}
\psi(P)=\max_{\substack{\sigma\in\Sigma\\\dim\sigma=n}}L_\sigma(P)
\quad\text{and}\quad
\psi(P)=L_\sigma(P)\Longleftrightarrow P\in\sigma.
\end{equation}
It follows that $\psi$ is convex, homogeneous and subadditive.

(3) Suppose that $\psi(P)=\sum_i\psi(X_i)$, and take an
$n$-dimensional cone $\sigma$ containing $P$. Then
\[
0=\sum_{i=1}^a\bigl(\psi(X_i)-L_\sigma(X_i)\bigr).
\]
Each summand is nonnegative. Hence $X_i\in\sigma$ for every $i$.
For $v\notin\sigma_P$, the coordinates within $\sigma$ give
\[
0=\lambda_v(P)=\sum_{i=1}^a\lambda_v(X_i).
\]
Thus $X_i\in\sigma_P$ for every $i$. The converse follows from
the linearity of $\psi$ on $\sigma_P$.
\end{proof}

\begin{notation}\label{not:residue-class}
Fix a function $\psi$ as in Lemma~\ref{lem:cone-weights}.
This choice depends only on $\Delta$ and the fixed subdivision.
For the construction in Sections~\ref{sec:auxiliary-prime}
and~\ref{sec:principal-minor}, fix a residue class
$r\in(\ZZ/D_\Sigma\ZZ)^\times$ and consider primes
\[
p\equiv r\pmod{D_\Sigma}.
\]
\end{notation}

\section{An auxiliary prime}
\label{sec:auxiliary-prime}

By Dirichlet's theorem, choose an auxiliary prime $\ell$ such that
\[
\ell>m_\Delta
\quad\text{and}\quad
\ell\equiv r\pmod{D_\Sigma}.
\]
The prime $\ell$ remains fixed as $p$ varies.

\subsection{A coefficient of an exponential determinant}
\label{sec:auxiliary-coefficient}

\begin{lemma}\label{lem:aux-cartier-jacobian}
Let $\ell$ be a prime and put
$R:=\FF_\ell[z_1,\ldots,z_m]$. For a polynomial
$G=\sum_{\underline u}c_{\underline u}\underline z^{\underline u}$,
define
\[
\mathcal C_\ell(G):=
\sum_{\underline b\in\ZZ_{\ge0}^m}
c_{\ell\underline b+(\ell-1)\underline1}
\underline z^{\underline b}
\quad\text{and}\quad
\underline1:=(1,\ldots,1).
\]
This is the Cartier operator on the coefficients of polynomial top forms.
For $F_1,\ldots,F_m\in R$, let
\[
J(\underline F):=
\det\left(\frac{\partial F_i}{\partial z_j}\right)_{1\le i,j\le m}.
\]
Then
\begin{equation}\label{eq:aux-cartier-jacobian}
\mathcal C_\ell\left(
J(\underline F)\prod_{i=1}^mF_i^{\ell-1}
\right)=J(\underline F).
\end{equation}
\end{lemma}

\begin{proof}
Write $\Omega_R^\bullet$ for the polynomial de Rham complex of $R$ over
$\FF_\ell$. Square brackets denote the cohomology class of a closed
form. For $F\in R$, put
\[
\delta(F):=[F^{\ell-1}dF]\in H^1(\Omega_R^\bullet).
\]
The form $F^{\ell-1}dF$ is closed. Direct differentiation gives
\[
\delta(FG)=F^\ell\delta(G)+G^\ell\delta(F).
\]
It also gives
\begin{equation*}
(F+G)^{\ell-1}d(F+G)-F^{\ell-1}dF-G^{\ell-1}dG=
d\left(
\sum_{j=1}^{\ell-1}\frac{(-1)^{j-1}}jF^jG^{\ell-j}
\right).
\end{equation*}
Consequently, $\delta$ is additive. Applying these two identities to
monomials and then to polynomials yields
\begin{equation}\label{eq:aux-delta-derivative}
\delta(F)=
\sum_{j=1}^m
\left(\frac{\partial F}{\partial z_j}\right)^\ell
[z_j^{\ell-1}dz_j].
\end{equation}
Multiplication by an $\ell$th power is well defined on de Rham cohomology,
since its differential is zero.

Taking the exterior product of \eqref{eq:aux-delta-derivative} for
$F_1,\ldots,F_m$, we obtain
\begin{equation*}
\left[
J(\underline F)\prod_{i=1}^mF_i^{\ell-1}
\,dz_1\wedge\cdots\wedge dz_m
\right]=
\left[
J(\underline F)^\ell
\prod_{j=1}^mz_j^{\ell-1}
\,dz_1\wedge\cdots\wedge dz_m
\right].
\end{equation*}
Thus the difference of these two top forms is exact.

For every $G\in R$ and every $j$, the definition of $\mathcal C_\ell$
implies
\[
\mathcal C_\ell\left(\frac{\partial G}{\partial z_j}\right)=0.
\]
Indeed, a monomial whose exponent of $z_j$ is congruent to $\ell-1$
could only arise by differentiating a monomial whose exponent of $z_j$
is divisible by $\ell$. Its derivative has coefficient zero.
It follows that $\mathcal C_\ell$ maps the coefficient of every
exact top form to zero. The definition also implies
\[
\mathcal C_\ell(A^\ell B)=A\mathcal C_\ell(B)
\qquad(A,B\in R).
\]
Applying $\mathcal C_\ell$ to the coefficients of the two top forms
gives \eqref{eq:aux-cartier-jacobian}.
\end{proof}

\begin{lemma}\label{lem:aux-ray-coefficient}
Let $m\ge1$ and let $\ell>m$ be a prime. Define
\[
h_k(\aa):=[z^k]
\exp\left(\sum_{j=1}^ma_jz^j\right)
\qquad(k\ge0),
\]
and put $h_k:=0$ for $k<0$. Then
\begin{equation}\label{eq:aux-ray-coefficient}
\left[\prod_{j=1}^ma_j^{\ell-1}\right]
\det(h_{\ell i-j})_{1\le i,j\le m}
\equiv\frac1{((\ell-1)!)^m}\pmod\ell.
\end{equation}
The coefficient on the left belongs to $\ZZ_{(\ell)}$.
\end{lemma}

\begin{proof}
Put
\[
\mathcal V(\underline z):=\det(z_i^{j-1})_{1\le i,j\le m}
=\prod_{i<j}(z_j-z_i)
\quad\text{and}\quad
F_j(\underline z):=\sum_{i=1}^mz_i^j.
\]
Expanding the determinant gives
\begin{align*}
\det(h_{\ell i-j})_{1\le i,j\le m}
&=
\left[\prod_{i=1}^mz_i^{\ell i-1}\right]
\mathcal V(\underline z)
\prod_{i=1}^m
\exp\left(\sum_{j=1}^ma_jz_i^j\right)\\
&=
\left[\prod_{i=1}^mz_i^{\ell i-1}\right]
\mathcal V(\underline z)
\exp\left(\sum_{j=1}^ma_jF_j(\underline z)\right).
\end{align*}
Therefore the coefficient in \eqref{eq:aux-ray-coefficient} is
\begin{equation}\label{eq:aux-ray-vandermonde}
\frac1{((\ell-1)!)^m}
\left[\prod_{i=1}^mz_i^{\ell i-1}\right]
\mathcal V(\underline z)\prod_{j=1}^mF_j(\underline z)^{\ell-1}.
\end{equation}
In particular, it belongs to $\ZZ_{(\ell)}$.

Reduce the polynomials in \eqref{eq:aux-ray-vandermonde} modulo $\ell$.
The Jacobian of $F_1,\ldots,F_m$ satisfies
\[
J(F_1,\ldots,F_m)=m!\mathcal V(\underline z).
\]
Since $\ell>m$, Lemma~\ref{lem:aux-cartier-jacobian} gives
\[
\mathcal C_\ell\left(
\mathcal V(\underline z)\prod_{j=1}^mF_j(\underline z)^{\ell-1}
\right)=\mathcal V(\underline z).
\]
The coefficient of $\prod_i z_i^{\ell i-1}$ on the left before applying
$\mathcal C_\ell$ is the coefficient of $\prod_i z_i^{i-1}$ in
$\mathcal V(\underline z)$, which is $1$. Substituting this into
\eqref{eq:aux-ray-vandermonde} proves the congruence.
\end{proof}

\begin{proposition}\label{prop:aux-lattice-coefficient}
Let $\SS\subset\ZZ^n\setminus\{0\}$ be a finite nonempty set.
Assume that, for every positive ray $\rho$ meeting $\SS$, there are a
primitive lattice vector $P_\rho$ and an integer $m_\rho\ge1$ such that
\[
\SS\cap\rho=\{P_\rho,2P_\rho,\ldots,m_\rho P_\rho\}.
\]
Let $\ell>\max_\rho m_\rho$ be a prime. For $P,Q\in \SS$, define
\[
g_{P,Q}(\aa):=
[\xx^{\ell P-Q}]
\exp\left(\sum_{R\in \SS}a_R\xx^R\right).
\]
Then
\begin{equation}\label{eq:aux-lattice-coefficient}
\left[\prod_{R\in \SS}a_R^{\ell-1}\right]
\det(g_{P,Q})_{P,Q\in \SS}
\equiv\frac1{((\ell-1)!)^{\#\SS}}
\not\equiv0\pmod\ell.
\end{equation}
The coefficient on the left belongs to $\ZZ_{(\ell)}$.
\end{proposition}

\begin{proof}
Interpret the exponential as a formal series in the parameters $a_R$.
Each parameter coefficient is a scalar multiple of a Laurent monomial
in $\xx$. Hence $g_{P,Q}$ is well defined.

A term contributing to the coefficient in
\eqref{eq:aux-lattice-coefficient} consists of a permutation
$\tau$ of $\SS$ and nonnegative integers $u_{P,R}$ satisfying
\begin{equation}\label{eq:aux-lattice-assignment}
\ell P=\tau(P)+\sum_{R\in \SS}u_{P,R}R
\quad(P\in \SS)
\quad\text{and}\quad
\sum_{P\in \SS}u_{P,R}=\ell-1
\quad(R\in \SS).
\end{equation}
Its coefficient is
\[
\sgn(\tau)
\prod_{P,R\in \SS}\frac1{u_{P,R}!}.
\]
Every $u_{P,R}$ is at most $\ell-1$, which proves the asserted
$\ell$-integrality.

Write $\|\cdot\|$ for the Euclidean norm. For each row of
\eqref{eq:aux-lattice-assignment}, the triangle inequality gives
\[
\ell\|P\|\le
\|\tau(P)\|+\sum_{R\in \SS}u_{P,R}\|R\|.
\]
Summing over $P$ gives equality, since
\begin{equation*}
\sum_{P\in \SS}\left(
\|\tau(P)\|+\sum_{R\in \SS}u_{P,R}\|R\|
\right)=\sum_{P\in \SS}\|P\|+(\ell-1)\sum_{R\in \SS}\|R\|=\ell\sum_{P\in \SS}\|P\|.
\end{equation*}
Hence equality holds in every row. Equality in the Euclidean triangle
inequality implies that $\tau(P)$ and every $R$ with $u_{P,R}>0$
lie on the positive ray through $P$.

Therefore $\tau$ preserves each positive ray, and $u_{P,R}=0$ unless
$P$ and $R$ lie on the same positive ray. Ordering the rows and columns
by rays, we obtain the coefficient in
\eqref{eq:aux-lattice-coefficient} as the product of the corresponding
coefficients for $\SS\cap\rho$.
For one ray, write $a_j:=a_{jP_\rho}$. The resulting matrix is
$(h_{\ell i-j})_{1\le i,j\le m_\rho}$ from
Lemma~\ref{lem:aux-ray-coefficient}. That lemma gives a factor
$((\ell-1)!)^{-m_\rho}$ modulo $\ell$.
Multiplication over all rays proves
\eqref{eq:aux-lattice-coefficient}.
\end{proof}

The assertion also holds for $\SS=\varnothing$, with the empty
determinant equal to $1$.

\subsection{A nonzero rational coefficient}

Fix $0<t<1$, and put
\[
\TT:=\TT_{\le t},\quad
\SS:=\TT\setminus\{0\}
\quad\text{and}\quad
H:=\sum_{P\in \SS}w(P).
\]
Since $t<1$, we have $\#\SS<m_\Delta<\ell$ and
\[
0\le\lambda_v(P)\le w(P)<1
\qquad(P\in\SS,\ v\in V).
\]

For $P\in\SS$, $v\in V$ and an integer $j\ge0$, put
\[
\theta_v(P):=\{r\lambda_v(P)\}
\quad\text{and}\quad
b_{j,v}(P):=\lfloor j\lambda_v(P)\rfloor.
\]
For a variable $z$ and an integer $a\ge0$, write
\[
G_a(z):=\prod_{i=0}^{a-1}(z-i)
\quad\text{and}\quad
G_0(z):=1.
\]
For $P,Q\in\SS$ and
$\underline u=(u_R)_{R\in\SS}\in\ZZ_{\ge0}^{\SS}$, put
\[
k_v(P,Q,\underline u):=
\lambda_v(Q)+\sum_{R\in\SS}u_R\lambda_v(R)-\theta_v(P).
\]
Define the formal matrix $B^{(r)}$ by
\[
B^{(r)}_{P,Q}:=
\sum_{\substack{\underline u\in\ZZ_{\ge0}^{\SS}\\
k_v(P,Q,\underline u)\in\ZZ_{\ge0}\ (v\in V)}}
\left(\prod_{v\in V}G_{k_v(P,Q,\underline u)}\bigl(-\theta_v(P)\bigr)\right)
\prod_{R\in\SS}\frac{a_R^{u_R}}{u_R!}.
\]
The entries are formal power series in the parameters $a_R$.
If $k_v(P,Q,\underline u)$ is an integer, it is automatically
nonnegative because its defining expression is greater than $-1$.

For a nonzero term and $v\notin\sigma_P$, we have
$\lambda_v(P)=\theta_v(P)=0$.
Since $G_k(0)=0$ for $k>0$, we obtain
\[
0=k_v(P,Q,\underline u)
=\lambda_v(Q)+\sum_{R\in\SS}u_R\lambda_v(R).
\]
All coordinates are nonnegative. Hence $Q$ and every $R$ with
$u_R>0$ belong to $\sigma_P$. In particular,
$B^{(r)}_{P,Q}=0$ if $Q\notin\sigma_P$.

\begin{proposition}\label{prop:residual-nonzero}
The rational number
\[
c_r:=\left[\prod_{R\in \SS}a_R^{\ell-1}\right]\det B^{(r)}
\]
is nonzero. More precisely,
\[
c_r\equiv
\frac{\displaystyle\prod_{P\in \SS}\prod_{v\in V}b_{\ell,v}(P)!}
{((\ell-1)!)^{\#\SS}}\not\equiv0\pmod\ell.
\]
\end{proposition}

\begin{proof}
In this coefficient every $u_R$ is at most $\ell-1$. The denominators of the coordinates divide $D_\Sigma$, and $\ell\nmid D_\Sigma$. Thus $c_r$ is integral at $\ell$.

Since $\ell\equiv r\pmod{D_\Sigma}$, for every integer $k\ge0$ we have
\[
b_{\ell,v}(P)+\theta_v(P)=\ell\lambda_v(P)
\quad\text{and}\quad
G_k\bigl(-\theta_v(P)\bigr)\equiv
G_k\bigl(b_{\ell,v}(P)\bigr)\pmod\ell.
\]
Moreover, $0\le b_{\ell,v}(P)<\ell$. A term is nonzero modulo $\ell$ only if
\[
k_v(P,\tau(P),\underline u_P)\le b_{\ell,v}(P)
\qquad(P\in \SS,\ v\in V).
\]
The total exponent of each parameter $a_R$ is $\ell-1$.
Therefore, for each fixed $v$, we have
\[
\begin{aligned}
\sum_{P\in \SS}k_v(P,\tau(P),\underline u_P)
&=\sum_{P\in \SS}\lambda_v(\tau(P))
+\sum_{R\in \SS}(\ell-1)\lambda_v(R)
-\sum_{P\in \SS}\theta_v(P)\\
&=\sum_{P\in \SS}b_{\ell,v}(P).
\end{aligned}
\]
All the preceding inequalities are therefore equalities. Every term
that is nonzero modulo $\ell$ satisfies
\[
\ell P=\tau(P)+\sum_{R\in\SS}u_{P,R}R,
\]
and each row contributes the factor $\prod_{v\in V}b_{\ell,v}(P)!$.
On each positive ray, $\SS$ is an initial segment of the lattice points.
Conversely, the proof of Proposition~\ref{prop:aux-lattice-coefficient}
shows that, in every term contributing to
\eqref{eq:aux-lattice-coefficient}, $\tau(P)$ and all $R$ with
$u_{P,R}>0$ lie on the positive ray through $P$.
The coordinate equations on that ray give
$k_v(P,\tau(P),\underline u_P)=b_{\ell,v}(P)$ for every $v\in V$.
Thus Proposition~\ref{prop:aux-lattice-coefficient} gives the stated
congruence after multiplication by
$\prod_{P\in\SS}\prod_{v\in V}b_{\ell,v}(P)!$.
Every factorial in the numerator and denominator is nonzero modulo $\ell$.
\end{proof}

\section{Proofs of Theorem~\ref{thm:main-L} and Corollary~\ref{cor:interior}}
\label{sec:main-proof}

By Proposition~\ref{prop:dwork-reduction}, to prove
Theorem~\ref{thm:main-L} it is enough to find constants $C_\Delta$
and $p_0(\Delta)$ such that, for every prime $p>p_0(\Delta)$, one
nondegenerate polynomial $f_p$ over $\FF_{p^2}$ satisfies
\eqref{eq:required-minors} for all $1\le k\le K$.

\subsection{A coefficient of a principal minor}
\label{sec:principal-minor}

For a fixed $0<t<1$, put
\[
\TT:=\TT_{\le t},\quad \SS:=\TT\setminus\{0\}
\quad\text{and}\quad H:=\sum_{P\in\SS}w(P).
\]
Since $w(R)<1=w(v)$ for $R\in\SS$ and $v\in V$, we have
$\SS\cap V=\varnothing$.
Consider the Laurent polynomial
\[
f(\xx):=\sum_{v\in V}a_v\xx^v
+\sum_{R\in\SS}a_R\xx^R,
\]
where the coefficients $a_R$ for $R\in\SS\cup V$ are independent parameters.
Let $p>\ell$ be a prime with $p\equiv r\pmod{D_\Sigma}$. Put
\[
\Lambda_v:=\sum_{P\in \SS}\lambda_v(P)
\quad\text{and}\quad
\mathcal M_p:=\prod_{R\in \SS}a_R^{\ell-1}
             \prod_{v\in V}a_v^{(p-\ell)\Lambda_v}.
\]
Since $D_\Sigma\Lambda_v\in\ZZ$ and $D_\Sigma\mid(p-\ell)$,
the vertex exponents are nonnegative integers.
The total degree of $\mathcal M_p$ is
\[
\nu_p:=(p-\ell)H+(\ell-1)\#\SS
=(p-1)H+(\ell-1)(\#\SS-H).
\]
For this polynomial, put
\[
\widetilde N^{(p)}_{P,Q}:=
[\xx^{pP-Q}]
\prod_{v\in V}E_p(\pi a_v\xx^v)
\prod_{R\in \SS}E_p(\pi a_R\xx^R)
\qquad(P,Q\in \TT).
\]
We regard these entries as formal power series in $\pi$ and the parameters.

\begin{proposition}\label{prop:actual-coefficient}
Put
\[
U_p:=\prod_{P\in \SS}\prod_{v\in V}b_{p,v}(P)!.
\]
Then $U_p\in\ZZ_p^\times$ and
\[
U_p[\pi^{\nu_p}\mathcal M_p]
\det(\widetilde N^{(p)}_{P,Q})_{P,Q\in \TT}
\equiv c_r\pmod p.
\]
Consequently, for all sufficiently large primes $p\equiv r\pmod{D_\Sigma}$,
this coefficient is a $p$-adic unit.
\end{proposition}

\begin{proof}
A term contributing to this coefficient determines a permutation $\tau$
of $\TT$ and nonnegative integers $u_{P,R}$ and $e_{P,v}$ satisfying
\[
pP=\tau(P)+\sum_{R\in \SS}u_{P,R}R+\sum_{v\in V}e_{P,v}v,
\]
\[
\sum_{P\in \TT}u_{P,R}=\ell-1
\quad\text{and}\quad
\sum_{P\in \TT}e_{P,v}=(p-\ell)\Lambda_v.
\]
By the subadditivity of $\psi$, we have
\[
p\psi(P)\le\psi(\tau(P))
+\sum_{R\in\SS}u_{P,R}\psi(R)
+\sum_{v\in V}e_{P,v}\psi(v).
\]
Summing over $P\in\TT$ gives equality, since
\[
\begin{aligned}
&\sum_{P\in \TT}\psi(\tau(P))
+\sum_{P\in \TT}\sum_{R\in \SS}u_{P,R}\psi(R)
+\sum_{P\in \TT}\sum_{v\in V}e_{P,v}\psi(v)\\
&\qquad=\ell\sum_{P\in \SS}\psi(P)
+(p-\ell)\sum_{v\in V}\Lambda_v\psi(v)
=p\sum_{P\in \TT}\psi(P).
\end{aligned}
\]
Equality holds in every row. By Lemma~\ref{lem:cone-weights}(3),
$\tau(P)$, every $R$ with $u_{P,R}>0$ and every $v$ with $e_{P,v}>0$
belong to $\sigma_P$. In the row indexed by $0$, positivity of $\psi$ away from $0$ gives
\[
\tau(0)=0,\quad
u_{0,R}=0
\quad\text{and}\quad
e_{0,v}=0.
\]
Thus the row indexed by $0$ contributes $1$ to the coefficient of
$\pi^{\nu_p}\mathcal M_p$, although its off-diagonal entries may be
nonzero when $0$ is interior to $\Delta$.

For a nonzero row, additivity of the coordinates inside $\sigma_P$ gives
\[
e_{P,v}
=p\lambda_v(P)-\lambda_v(\tau(P))
-\sum_{R\in \SS}u_{P,R}\lambda_v(R).
\]
In particular,
\[
0\le e_{P,v}\le p\lambda_v(P)<p
\quad\text{and}\quad
0\le u_{P,R}\le\ell-1<p.
\]
The coefficients of $E_p(z)$ in degrees less than $p$ are the ordinary inverse factorials. Hence these row contributions are
\[
\prod_{v\in V}\frac{1}{e_{P,v}!}
\prod_{R\in \SS}\frac{1}{u_{P,R}!}.
\]
By the definition of $b_{p,v}(P)$,
\[
e_{P,v}=b_{p,v}(P)-k_v(P,\tau(P),\underline u_P),
\]
and multiplication by the row factorial gives
\[
\frac{b_{p,v}(P)!}{e_{P,v}!}
=G_{k_v(P,\tau(P),\underline u_P)}\bigl(b_{p,v}(P)\bigr)
\equiv G_{k_v(P,\tau(P),\underline u_P)}\bigl(-\theta_v(P)\bigr)\pmod p.
\]

Conversely, consider a nonzero term contributing to $c_r$.
The definition of $B^{(r)}$ implies that $\tau(P)$ and every $R$
with $u_{P,R}>0$ belong to $\sigma_P$. If the integers
\[
e_{P,v}:=b_{p,v}(P)-k_v(P,\tau(P),\underline u_P)
\]
are nonnegative, the row vector equations hold. Summing their coordinate equations over $P$ gives
\[
\sum_{P\in \SS}e_{P,v}=(p-\ell)\Lambda_v,
\]
as required. If $e_{P,v}<0$ for some $P$ and $v$, then $k_v>b_{p,v}(P)$ and
\[
G_{k_v}\bigl(b_{p,v}(P)\bigr)=0.
\]
The corresponding term of $c_r$ is therefore zero modulo $p$.
Summing over all terms proves the congruence.

Finally, $c_r$ is a fixed nonzero rational number by Proposition~\ref{prop:residual-nonzero}. Only finitely many primes divide its numerator or denominator. Also $b_{p,v}(P)<p$, which proves $U_p\in\ZZ_p^\times$.
\end{proof}

\subsection{Completion of the proofs}
\label{sec:cutoffs}

\begin{proof}[Proof of Theorem~\ref{thm:main-L}]
Use the weights $\rho_k$ and the numbers $H_k$ defined before
Proposition~\ref{prop:dwork-reduction}. Put
\[
 \SS_k:=\TT_{\le\rho_k}\setminus\{0\}\qquad(0\le k\le K).
\]
If $K=0$, take $C_\Delta:=0$. The argument over a finite field below,
with the product over $k$ empty, gives a nondegenerate polynomial
over $\FF_{p^2}$ for all sufficiently large $p$.
Proposition~\ref{prop:dwork-reduction} then proves the theorem.
We may therefore assume that $K\ge1$.
Since $\TT_{\le\rho_K}=\SS_K\cup\{0\}=\TT_{<1}$, the last principal
minor has order $m_\Delta$.

Use the fan $\Sigma$ and the integer $D_\Sigma$ fixed in
Section~\ref{sec:preliminaries}. For each unit residue class $r$
modulo $D_\Sigma$, choose an auxiliary prime $\ell_r$ as in
Section~\ref{sec:auxiliary-prime}.
Since $\#\SS_k<m_\Delta<\ell_r$, the same prime applies to every $k$.
For each $r$ and $1\le k\le K$, apply Proposition~\ref{prop:residual-nonzero}
with $\SS=\SS_k$ and $\ell=\ell_r$ to obtain a fixed nonzero rational
coefficient $c_{r,k}$. Define
\begin{equation}\label{eq:uniform-constant}
 C_\Delta:=\max_r(\ell_r-1)
             \sum_{P\in \SS_K}(1-w(P)).
\end{equation}
Choose $p_0(\Delta)$ to exclude primes dividing $D_\Sigma$, primes not
exceeding any $\ell_r$ and the finitely many primes dividing the
numerator or denominator of one of the $c_{r,k}$.

Fix a remaining prime $p\equiv r\pmod{D_\Sigma}$. Work in the full coefficient
space with one parameter $a_R$ for every nonzero
$R\in\Delta\cap\ZZ^n$. We may set the constant coefficient to zero,
since adding a constant does not change Newton slopes.
Let $\mathcal H_{p,k}$ be the coefficient of
\[
 \pi^{\nu_{p,k}}
\quad\text{where}\quad
\nu_{p,k}:=(p-\ell_r)H_k+(\ell_r-1)\#\SS_k,
\]
in the universal principal minor indexed by $\TT_{\le\rho_k}$, reduced modulo $p$.
The number $\nu_{p,k}$ is an integer because
$D_\Sigma H_k\in\ZZ$ and $D_\Sigma\mid(p-\ell_r)$.
The polynomial $\mathcal H_{p,k}$ has total degree at most $\nu_{p,k}$, and
\[
0\le\nu_{p,k}-(p-1)H_k
=(\ell_r-1)\sum_{P\in\SS_k}(1-w(P))\le C_\Delta.
\]
Proposition~\ref{prop:actual-coefficient} proves that
$\mathcal H_{p,k}\ne0$: setting all parameters outside $\SS_k$ and the
vertices to zero leaves the nonzero coefficient calculated there.

Choose $\mathcal R\in\ZZ[\aa]$ defining a nonempty principal open
subset of the nondegenerate locus over $\QQ$ \cite{AS89,LW09}.
Clearing denominators in the Nullstellensatz identities for the face
equations on this subset shows that $\mathcal R(\aa)\ne0$ implies nondegeneracy
for all but finitely many primes.
The polynomial
\[
 \overline{\mathcal R}\,
 \prod_{v\in\Vertices(\Delta)\setminus\{0\}}a_v
 \prod_{k=1}^K\mathcal H_{p,k}
\]
is nonzero over $\FF_p$ and has total degree $O_\Delta(p)$.
For sufficiently large $p$ this degree is less than $p^2$.
A nonzero polynomial of total degree smaller than the cardinality of
a finite field cannot vanish at every point of its affine space.
We can therefore choose one coefficient vector over $\FF_{p^2}$ at
which the displayed product is nonzero. It defines a nondegenerate
polynomial $f_p$ with Newton polytope $\Delta$.

For a determinant term of the principal minor indexed by $\TT_{\le\rho_k}$, the
total formal degree in $\pi$ is at least
\[
 \sum_{P\in \TT_{\le\rho_k}}w(pP-\tau(P))\ge(p-1)H_k.
\]
Thus the principal minor has no formal terms of degree less than
$\lceil(p-1)H_k\rceil$. At the Teichm\"uller lift of the chosen
coefficient vector, the formal coefficient at degree $\nu_{p,k}$
reduces modulo $p$ to the nonzero value $\mathcal H_{p,k}(\aa)$.
It is therefore a unit in the unramified coefficient ring.
For $p-1>C_\Delta$, we have
\[
0\le\nu_{p,k}-\lceil(p-1)H_k\rceil\le C_\Delta<p-1.
\]
Lemma~\ref{lem:formal-specialization} therefore gives
\[
 v_\pi\det N_{\TT_{\le\rho_k},\TT_{\le\rho_k}}\le\nu_{p,k}
 \le(p-1)H_k+C_\Delta.
\]

Increase $p_0(\Delta)$ to ensure $p-1>2DC_\Delta$.
The polynomial $f_p$ satisfies \eqref{eq:required-minors}
for every $1\le k\le K$ with the same constant $C_\Delta$.
Proposition~\ref{prop:dwork-reduction} proves
\eqref{eq:quantitative-L-height} and \eqref{eq:quantitative-L-slope},
as well as the asserted height estimate for $f_p$.
\end{proof}

\begin{proof}[Proof of Corollary~\ref{cor:interior}]
Since $0$ is interior to $\Delta$, \cite[Theorem~1.3(c)]{DL91} gives
purity of weight $n$ for all reciprocal roots of $L_f^*(s)^{(-1)^{n-1}}$.
For a complex embedding and a reciprocal root $\gamma$, purity gives
$\overline\gamma=q^n/\gamma$.
The coefficients of this polynomial belong to $\QQ(\zeta_p)$, in which $p$ has a
unique prime. Complex conjugation preserves the valuation of each
coefficient. Hence this polynomial and its conjugate have the same Newton
polygon, and their slope multiset is invariant under
$\alpha\mapsto n-\alpha$.

The generic Newton slopes have the same symmetry. Applying this
symmetry at an ordinary prime supplied by \cite[Theorem~1.6]{Wan04}
gives the corresponding identity for the Hodge slopes. Thus
\[
\alpha_{p,i}+\alpha_{p,V_\Delta+1-i}=n
\quad\text{and}\quad
\beta_i+\beta_{V_\Delta+1-i}=n.
\]
Hence
Theorem~\ref{thm:main-L} gives \eqref{eq:quantitative-L-slope}
for the slopes corresponding to Hodge slopes greater than $n-1$.

Suppose now that $n=2$.
The remaining Hodge slopes are all equal to one.
Both Newton and Hodge slope sums are $V_\Delta$.
For $e_p(x):=\GNP_{L,p}(\Delta)(x)-\HP_L(\Delta)(x)$,
symmetry gives
\[
e_p(V_\Delta-x)=e_p(x)\qquad(0\le x\le V_\Delta).
\]
By Theorem~\ref{thm:main-L},
$0\le e_p(x)\le C_\Delta/(p-1)$ on $[0,m_\Delta]$ and $[V_\Delta-m_\Delta,V_\Delta]$.
The Hodge polygon is linear on $[m_\Delta,V_\Delta-m_\Delta]$. Convexity of the Newton
polygon therefore gives
\[
0\le e_p(x)\le e_p(m_\Delta)\le\frac{C_\Delta}{p-1}
\qquad(m_\Delta\le x\le V_\Delta-m_\Delta).
\]
This proves \eqref{eq:quantitative-L-height} on $[0,V_\Delta]$.
Since $\alpha_{p,i}-\beta_i=e_p(i)-e_p(i-1)$,
\eqref{eq:quantitative-L-slope} holds for every $1\le i\le V_\Delta$.
Letting $p\to\infty$ proves the uniform convergence.
\end{proof}

\begin{remark}\label{rem:cutoff-factor}
If $m_\Delta<V_\Delta$, the Hodge lower bound at $m_\Delta+1$
and \eqref{eq:quantitative-L-height} give
$\alpha_{p,m_\Delta+1}\ge1-C_\Delta/(p-1)$.
Thus, for any fixed $c$ with $\beta_{m_\Delta}<c<1$, exactly
$m_\Delta$ generic slopes are below $c$ for sufficiently large $p$.
\end{remark}

\begin{remark}
In dimension one, Theorem~\ref{thm:main-L} gives convergence of the whole
Newton polygon for every integral interval containing zero. If zero is
an endpoint, all finite Hodge slopes are smaller than one. If zero
is interior, the only missing slope is the final slope one, which is
determined by the common endpoint. The polynomial case was proved
in \cite{Zhu03,Zhu04}. For the Laurent polynomial case, see
\cite[Section~3]{Blache11}.
\end{remark}

\section{Examples without the semigroup hypothesis}
\label{sec:obstruction}

We begin with a simplex whose lattice-point semigroup has infinite
complement and whose Hodge polygon has slope $1/2$.

\begin{example}\label{ex:nontrivial-low-slopes}
Put
\[
 v:=(1,1,3)
\quad\text{and}\quad
\Delta:=\conv\{0,2e_1,e_2,v\}\subset\RR^3.
\]
The coordinates of $P=(P_1,P_2,P_3)$ with respect to $2e_1,e_2$ and $v$ are
\[
\lambda_{2e_1}(P)=\frac{P_1}2-\frac{P_3}6,\quad
\lambda_{e_2}(P)=P_2-\frac{P_3}3
\quad\text{and}\quad
\lambda_v(P)=\frac{P_3}3.
\]
Requiring these coordinates to lie in $[0,1)$ shows that the lattice points in the
half-open fundamental parallelepiped are
\[
 0,\quad e_1,\quad (1,1,1),\quad (2,1,1),\quad (1,1,2)
\quad\text{and}\quad (2,1,2).
\]
Their weights are $0$, $1/2$, $4/3$, $11/6$, $7/6$ and $5/3$,
respectively. Every point of $\MM(\Delta)$ is uniquely a sum of one
of these points and a nonnegative integral combination of
$2e_1,e_2$ and $v$. Hence
\[
 \sum_{P\in\MM(\Delta)}z^{w(P)}
 =\frac{1+z^{1/2}+z^{7/6}+z^{4/3}+z^{5/3}+z^{11/6}}{(1-z)^3}.
\]
It follows that the finite Hodge slopes are
\[
 0,\quad\frac12,\quad\frac76,\quad\frac43,\quad\frac53
\quad\text{and}\quad\frac{11}{6}.
\]
The same calculation gives
\[
 \Delta\cap\ZZ^3=\{0,e_1,2e_1,e_2,v\}.
\]
Every point in the semigroup generated by this set has third
coordinate divisible by three. However,
$(1,1,1)+2m e_1$ belongs to $\MM(\Delta)$ for every integer $m\ge0$.
Thus the semigroup has infinite complement in $\MM(\Delta)$.

Here $\TT_{<1}=\{0,e_1\}$ and $m_\Delta=2$.
Theorem~\ref{thm:main-L} gives a uniform error of order $p^{-1}$
on $[0,2]$. In particular,
\[
 \lim_{p\to\infty}\alpha_{p,1}=0
\quad\text{and}\quad
\lim_{p\to\infty}\alpha_{p,2}=\frac12.
\]
\end{example}

The following family shows that convergence of the full Newton polygon
can fail in every dimension $n\ge3$. Compare Wan's four-dimensional
example in \cite[Examples~1.11]{Zhu12}.

\begin{proposition}\label{prop:counterexample}
For $n\ge3$, put
\[
 v:=(1,\ldots,1,3)\in\ZZ^n
\quad\text{and}\quad
\Delta_n:=\conv\{0,e_1,\ldots,e_{n-1},v\}.
\]
The finite Hodge slopes are
\[
 0,\quad
\frac{n+1}{3}
\quad\text{and}\quad
\frac{2n-1}{3}.
\]
For every prime $p\equiv2\pmod3$, all nondegenerate polynomials with
this Newton polytope have finite Newton slopes
\[
 0,\quad
\frac n2
\quad\text{and}\quad
\frac n2.
\]
Therefore the generic Newton polygon does not converge to the Hodge
polygon as $p$ tends to infinity through all primes.
\end{proposition}

\begin{proof}
The nonzero vertex matrix has determinant $3$. The half-open
fundamental parallelepiped has three lattice points, whose vertex
coordinates are
\[
 (0,\ldots,0),\quad
(2/3,\ldots,2/3,1/3)
\quad\text{and}\quad
(1/3,\ldots,1/3,2/3).
\]
Their weights are $0$, $(2n-1)/3$ and $(n+1)/3$.
Every cone lattice point has a unique representation as one of these
three points plus a nonnegative integral combination of the vertices.
It follows that
\[
 \sum_{P\in \MM(\Delta_n)}z^{w(P)}
 =\frac{1+z^{(n+1)/3}+z^{(2n-1)/3}}{(1-z)^n}.
\]
This proves the formula for the Hodge slopes. Since both nonzero weights in the
numerator are greater than one for $n\ge3$, the only lattice points
in $\Delta_n$ are its vertices and the origin.

A constant term multiplies the reciprocal roots by a root of unity
and does not change their valuations. After omitting it, every
polynomial in this family is diagonal:
\[
 f=a_1x_1+\cdots+a_{n-1}x_{n-1}
       +a_nx_1\cdots x_{n-1}x_n^3
\quad\text{with}\quad
a_1\cdots a_n\ne0.
\]
For $p\ne3$, linear independence of every subset of the vertex
columns over $\FF_p$ shows nondegeneracy on every face not containing
zero.

For a diagonal polynomial, the Gauss-sum formula computes each
Newton slope as the average weight on the corresponding orbit of
multiplication by $p$ on the fundamental parallelepiped
\cite[Corollary~2.3 and equation~(20)]{Wan04}. After a finite extension, a diagonal change of variables makes every
nonzero coefficient equal to one without changing normalized slopes.
When $p\equiv2\pmod3$,
multiplication by $p$ interchanges the two nonzero coordinate vectors.
Their average weight is
\[
 \frac12\left(\frac{n+1}{3}+\frac{2n-1}{3}\right)=\frac n2.
\]
The zero vector contributes slope zero. When $p\equiv1\pmod3$, all
three coordinate vectors are fixed and the Newton polygon equals the
Hodge polygon. Both residue classes contain infinitely many primes
and give different Newton polygons.
\end{proof}

For $n=3$, the height difference at $x=2$ for primes
$p\equiv2\pmod3$ is
\[
 \frac32-\frac43=\frac16.
\]
Every point in the semigroup generated by $\Delta_n\cap\ZZ^n$ has
last coordinate divisible by three, whereas infinitely many lattice
points in the cone do not. Thus this family does not satisfy the
finite-complement hypothesis in Wan's conjecture.

\begin{remark}
For $n=4$ and $p\equiv2\pmod3$, the characteristic-series polygons agree at
$x=m_\Delta=1$ and differ at $x=6$. Indeed,
\[
 \TT_{<1}=\{0\}
\quad\text{and}\quad
\TT_{<2}=\{0,e_1,e_2,e_3,v,(1,1,1,2)\}.
\]
The weight sum on $\TT_{<2}$ is $17/3$. The first six characteristic-series
Newton slopes for $p\equiv2\pmod3$ are
$0$, $1$, $1$, $1$, $1$ and $2$, whose sum is $6$.
Thus the height difference at $x=6$ is $1/3$.
\end{remark}

\end{document}